\documentclass[11pt,reqno]{amsart}
\usepackage {amsmath,amssymb,verbatim}
\usepackage[pagebackref,pdftex,hyperindex]{hyperref}

\usepackage[margin=3cm]{geometry}

\newcommand{\C}{\mathbb{C}}

\newcommand{\N}{\mathbb{N}}

\renewcommand{\P}{\mathbb{P}}
\newcommand{\Q}{\mathbb{Q}}
\newcommand{\R}{\mathbb{R}}
\renewcommand{\S}{\mathbb{S}}

\newcommand{\Z}{\mathbb{Z}}

\newcommand{\fg}{\mathfrak{g}}

\newcommand{\fP}{\mathfrak{P}}

\newcommand{\cC}{\mathcal{C}}

\newcommand{\cE}{\mathcal{E}}

\newcommand{\cH}{\mathcal{H}}

\newcommand{\cL}{\mathcal{L}}

\newcommand{\cS}{\mathcal{S}}

\newcommand{\cX}{\mathcal{X}}

\renewcommand{\a}{\alpha}
\renewcommand{\b}{\beta}

\renewcommand{\d}{\delta}
\newcommand{\e}{\varepsilon}
\newcommand{\s}{\sigma}

\renewcommand{\phi}{\varphi}

\newcommand{\ie}{{\rm i.e.\ }} 

\renewcommand{\i}{\sqrt{-1}}
\renewcommand{\leq}{\leqslant}
\renewcommand{\geq}{\geqslant}
\newcommand{\abs}[1]{\left\lvert#1\right\rvert}
\newcommand{\norm}[1]{\left\|#1\right\|}

\newcommand{\HS}{\mathrm{HS}}

\newcommand{\NA}{\mathrm{NA}}

\DeclareMathOperator{\area}{area}

\DeclareMathOperator{\Aut}{Aut}

\DeclareMathOperator{\dd}{\sqrt{-1}\partial \bar{\partial}}

\DeclareMathOperator{\GL}{GL}

\DeclareMathOperator{\Hom}{Hom}
\DeclareMathOperator{\id}{id}

\DeclareMathOperator{\vol}{vol}

\numberwithin{equation}{section}       

\newtheorem{prop} {Proposition} [section]
\newtheorem{thm}[prop] {Theorem} 
\newtheorem{dfn}[prop] {Definition}
\newtheorem{lem}[prop] {Lemma}

\newtheorem{rem}[prop]{Remark}
\theoremstyle{remark}
\newtheorem*{ackn}{\bf{Acknowledgment}}

\newtheorem*{thmA}{\bf{Theorem A}} 
\newtheorem*{thmB}{\bf{Theorem B}} 
\newtheorem*{thmC}{\bf{Theorem C}} 
\newtheorem*{thmD}{\bf{Theorem D}}

\newtheorem*{dfn*}{\bf{Definition}}

\title[Uniform stability for equivariant polarizations]{Uniform stability and coercivity for equivariant polarizations} 
\date{\today} 

\author{Tomoyuki Hisamoto}
\address{Graduate School of Mathematics\\
  Nagoya University\\
  Furocho\\
  Chikusa\\
  Nagoya\\ 
  Japan}
\email{hisamoto@math.nagoya-u.ac.jp}


\begin{document}

\begin{abstract}
We introduce uniform K-stability and its relationship with the coercivity property of the K-energy functional, for general polarized manifolds. 
Since the automorphism groups are not necessarily finite, size of the norm measuring uniformity should be reduced with respect to the group action. 
About this point we explain that it is enough to take the reduced norm for a single sub-torus, actually the center, in the cscK problem. 
Our main theorem then describes the slope of the reduced J-functional along any torus-equivariant test configuration. 
In the toric case it is shown that the uniform stability is indeed equivalent to the coercivity of the K-energy.  
In the Fano manifolds case existence of the KE metric implies the uniform stability. 

\end{abstract}

\maketitle

\setcounter{tocdepth}{1}
\tableofcontents 

\section*{Introduction}

The background of this paper is so-called Yau-Tian-Donaldson conjecture, which states that a polarized manifold admits a constant scalar curvature K\"ahler metric  if and only if it is K-polystable. 
To ensure the metric existence, the idea of {\em uniform K-stability} was first introduced by the thesis \cite{Sze06} and developed as it can be seen in \cite{Der15}, \cite{Der16},  \cite{BHJ15}, and \cite{BHJ16}.  
In analytic point of view the K-energy functional is defined over the space of K\"ahler metrics in the prescribed cohomology class and characterizes the cscK property as the critical point. The uniform K-stability then should be translated into the quantitative growth condition, so-called {\em coercivity} of the K-energy. 
Throughout the paper we denote a polarized manifold by $(X, L)$ and 
let $\cH$ be the collection of positively curved fiber metrics on $L$. 

\begin{dfn*}
We say that the K-energy functional $M\colon \cH \to \R$ is coercive (with respect to Aubin's $J$-functional) if there exists a constant $\e, C >0$ such that 
\begin{equation*}
M(\phi) \geq \e J(\phi) -C 
\end{equation*}
holds for any $\phi \in \cH$. 
\end{dfn*}  

Coercivity concept originates from Aubin's strong Moser-Trudinger inequality on the two-sphere. Relation with the existence of a K\"ahler-Einstein metric was first discussed by \cite{Tian97}. This condition assures a minimizer in the appropriate completion of $\cH$ and in the Fano case the weak minimizer in fact defines a smooth K\"ahler-Einstein metric (see \cite{BBGZ13}). 
On the other hand, the new technique exploited by \cite{DR17}, \cite{BDL16} shows that such coercivity in fact follows from the existence of a smooth cscK metric. 
See also Theorem \ref{metric to coercivity} in our framework. 

In the algebraic geometry side a ray on $\cH$ can be regarded as a $\C^*$-equivariant flat family of polarized schemes $(\cX, \cL)$ whose generic fiber is isomorphic to $(X, L)$. The geometric family $(\cX, \cL)$ is called a test configuration. Actually following our previous work \cite{BHJ15} one can define the non-Archimedean analogue $M^{\NA}(\cX, \cL), J^{\NA}(\cX, \cL) \in\Q$ of the classical energy functionals to have 
\begin{equation*}
M^{\NA}(\cX, \cL) = \lim_{t \to \infty} M(\phi^t)/t 
\end{equation*}
for some smooth ray $\phi^t$ compatible with the test configuration ({\em e.g.} for $\phi^t$  obtained by pulling back a Fubini-Study type fiber metric of $\cL$ by the $\C^*$-action). 
In this terminology, Donaldson's definition of K-stability is equivalent to require $M^{\NA}(\cX, \cL)>0$ for any non-trivial $(\cX, \cL)$ and $(X, L)$ is called uniformly K-stable if there exists a constant $\e>0$ which satisfies 
\begin{equation*}
M^{\NA}(\cX, \cL) \geq \e J^{\NA}(\cX, \cL)
\end{equation*}
for any $(\cX, \cL)$. 
It has been clarified that this a priori stronger notion of stability fits together with recently developed algebraic methods to check the stability condition. 
The uniform estimate even helps seeking a cscK metric. We refer to \cite{BBJ15} for the detail of this direction. 

Now, the above definition of uniform stability implies that the automorphism group $\Aut(X, L)$ is finite. When the group is infinite, the required theory has been established only for the toric manifolds. In this paper we extend the definition to general polarized manifold and derive the stability from the coercivity, extending of the slope formula in our previous work \cite{BHJ16}. 
This corresponds to the following coercivity. 

\begin{dfn*}
Let $G=K^\C \subseteq \Aut(X, L)$ be a reductive algebraic subgroup and $K$ be a maximal compact subgroup of $G$. 
We denote the center by $C(G)$. 
Any $g \in \Aut(X, L)$ pull-backs $\phi \in \cH$ to $\phi_g$. 
We say that a functional $F\colon \cH \to \R$ is $G$-coercive if there exists a constant $\e, C >0$ such that 
\begin{equation*}
F(\phi) \geq \e \inf_{g \in C(G)} J(\phi_g) -C 
\end{equation*}
holds for any $\phi \in \cH^K$. 
\end{dfn*}  

One point in the above definition is that we took the infimum for $g \in C(G)$. 
Inspired by the work \cite{DR17}, we explain that this relatively small choice $C(G)$ is sufficient to handle with the cscK problem and also that taking $G$ as a single maximal torus can not control the energy. 
In particular for the Fano case we use the D-energy (introduced in the work of Bando-Mabuchi, Ding) and show: 

\begin{thmA}\label{coercivity to metric in the Fano case} 
Let $X$ be a Fano manifold. Assume that D-energy functional is coercive for a reductive subgroup $G=K_\C \subseteq \Aut(X, -K_X)$. Then $X$ admits a $K$-invariant K\"ahler-Einstein metric in the first Chern class. 
\end{thmA}

Based on these considerations, to introduce the uniform K-stability we first fix a sub-torus $T \subseteq \Aut(X, L)$. For the metric problem one should take $T=C(G)$, however, the main result in the below (Theorem B) is valid for arbitrary torus. 
As long as a test configuration $(\cX, \cL)$ is $T$-equivariant, one may define for any one-parameter subgroup $\mu \in N:= \Hom(\C^*, T)$ the new test configuration $(\cX_\mu, \cL_\mu)$ twisted by $\mu$. 
The definition naturally extends to $N_\Q=N \otimes \Q$. 
We then define the reduced non-Archimedean J-functional 
\begin{equation*}
J_T^\NA(\cX, \cL ) := \inf_{\mu \in N_\Q} J^\NA (\cX_\mu, \cL_\mu). 
\end{equation*} 
On the other hand, the reduced J-energy is naturally defined for $\phi \in \cH^S$ as 
$J_T(\phi) := \inf_{g \in T} J(\phi_g)$.  
Let us now take an $S$-invariant fiber metric $\Phi$ on $\cL$ so that the $\C^*$-action pull-backs $\Phi$ to define the ray $\phi^t \in \cH$ $(t \in [0, \infty))$.  
For the other metric $\Psi$ the difference $\abs{\phi^t -{\psi}^t}$ is bounded uniformly in $t$ so that the choice of $\Phi$ does not matter. 
In the following we may even take a weak (namely non-smooth) geodesic ray associated with the test configuration. 

\begin{thmB}
Let $(X, L)$ be a general polarized manifold. 
For any fiber metric $\Phi$ on a $T$-equivariant test configuration $(\cX, \cL)$ we have   
\begin{equation*}
J_T(\cX, \cL) = \lim_{t \to \infty} \frac{J_T(\phi^t)}{t}. 
\end{equation*} 
\end{thmB}  

It immediately implies: 

\begin{thmC}
If K-energy functional is G-coercive, then the polarized manifold is uniformly K-stable for $G$. 
\end{thmC}


The expected correspondence between the uniform stability and the coercivity is completed for toric polarized manifolds. 
In this setting we may take the maximal torus $T$, the compact sub-torus $S=\Hom (\C^*, T) \otimes \S^1$, and restrict ourselves to the space of invariant metrics $\cH^S$, according to the symmetry of this class of manifolds. A toric test configuration is represented by a convex, rational piecewise-linear function $f$ on the moment polytope $P \subset M_{\R}$. Following \cite{Don02} we define 
\begin{equation*}
L(f):= \int_{\partial P} f - \frac{\area(\partial P)}{\vol(P)}\int_P f, 
\end{equation*}
which is precisely the toric interpretation of the non-Archimedean K-energy. 
In addition, we introduce a new invariant {\em J-norm} as 
\begin{equation*}
\norm{f}_J := \inf_{\ell}\bigg\{  \frac{1}{\vol(P)}\int_P (f+\ell) -\min_P \{ f+\ell\} \bigg\}, 
\end{equation*}
where $\ell$ runs through all the affine functions. This interprets the non-Archimedean J-functional. 

\begin{thmD}
For any toric polarized manifold with the maximal torus $T$, K-energy functional is $T$-coercive in $\cH^{S}$ if and only if there exists a constant $\e>0$ such that 
\begin{equation*}
L(f) \geq \e \norm{f}_J 
\end{equation*} 
holds for any convex, rational piecewise-linear function $f\colon P\to \R$. 

\end{thmD}

Standing on other (a priori not clear to be equivalent) definitions of coercivity the kind of results seems to be known to experts. 
(See the the comments after Remark \ref{constant}.)
We give a detail account of the above result. The discussion compares the several definitions of toric stability and coercivity, and clarifies the relationship with the general picture illustrated in Theorem C. 
The proof of the coercivity originates from \cite{Don02}, \cite{ZZ08b} where they adopt the larger ``boundary norm" to measure the uniformity. 

The organization of the paper is as follows: 
After preliminary materials we first show Theorem B for a general polarization. 
Secondly we look at the toric setting and derive the coercivity from stability. 
In the last we discuss the stability and coercivity for general polarizations, particularly pointing out validity and convenience of our (possibly small) torus setting. 
 


\begin{ackn}
The author express his gratitude to Professor R. Berman. Quite a few ideas for the proof of Theorem B are based on the communication with him.   
The author would like to thank Professor S. Boucksom and M. Jonsson for fruitful discussions. He is supported by JSPS KAKENHI Grant Number 15H06262. 
\end{ackn} 


\section{Preliminary materials}

Throughout the paper $(X, L)$ denotes an $n$-dimensional polarized algebraic manifold, defined over the complex number field $\C$.  
Any K\"ahler metric $\omega$ is chosen to be in the first Chern class $2\pi c_1(L)$. 
The volume $V:= L^n$ is then represented as $\int_X \omega^n$.  
Define $\hat{S}:=-\frac{n}{2}V^{-1}K_XL^{n-1}$ which equals to the mean value of the scalar curvature. 

\subsection{Energy functionals} 
We adopt here the additive notation for the fiber metric of the holomorphic line bundle. In each local trivialization neighborhood one may take a function $\phi$ so that the fiber metric is given by the multiplication of $e^{-\phi}$. 
In the additive notation we write the metric by $\phi$, ignoring the dependence to the local trivialization. 
The curvature $\dd \phi$ is actually independent of those trivializations.  
Any K\"ahler metric in $2\pi c_1(L)$ may be written as $\omega=\omega_\phi= \dd \phi$ for some fiber metric $\phi$ which is unique up to a constant. 
It is convenient to write the things in terms of $\phi$. 
For this reason, we introduce $\cH$ as the set of  smooth fiber metric $\phi$ with $\dd \phi$ positive definite and identify the K\"ahler metric with the element in $\cH$. For example the scalar curvature is written as $S_\phi=S_{\omega_\phi}$. 
It is also necessary to consider the {\em singular} fiber metric, which consists of $L^1$ weights $\phi$ such that the curvature current $\dd \phi$ is semipositive. 
It means that in each local trivialization $\phi$ is identified with the plurisubharmonic (psh for short) function.  

To find out the canonical K\"ahler metrics in $2\pi c_1(L)$, there are the two main energy functionals defined on $\cH$. 
One is Aubin-Mabuchi energy $E$ and the other is Mabuchi's K-energy $M$. 
The two energies in $\cH$ are characterized by their differential:  
\begin{equation*} 
\d E(\phi) = \frac{1}{V}\int_X (\d \phi) {\omega_\phi^n}  \ \ \  \text{and} 
\end{equation*} 
\begin{equation*}
\d M(\phi) = - \frac{1}{V}\int_X (\d \phi) (S_{\phi} -\hat{S}) {\omega_\phi^n}.  
\end{equation*} 
These forms naturally appear in the asymptotic equivariant Riemann-Roch thoerem where $\d \phi$ is replaced to some Hamilton function.  
Further, taking a reference metric $\psi$ we define Aubin's J-functional  \begin{equation*}
J(\phi) = J_\psi(\phi) =\frac{1}{V} \int_X (\phi-\psi) \omega_\psi^n -E(\phi) \geq 0.  
\end{equation*} 
Since $J$ is scale-free: $J(\phi +c) =J(\phi)$, it decends to a functional in $\omega$. 
Following \cite{BEGZ10}, the definition of $E$ and $J$ further extends to a singular $\phi$ once we consider $\omega_\phi^n$ as the {\em non-pluripolar Monge-Amp\`ere product}. It is the natural extension of Bedfod-Taylor's product for bounded psh functions, so that the measure $\omega_\phi^n$ drops the mass on the unbounded locus.   
Another normalization we may take is 
\begin{equation*}
J'(\phi) :=  \sup_X (\phi-\psi) - E(\phi). 
\end{equation*}
It holds ({\em e.g.} from Proposition $2.7$ of \cite{GZ05}) that 
\begin{equation*}
J(\phi) \leq J'(\phi) \leq J(\phi)+C 
\end{equation*}
for some constant $C$ depends on $\psi$.  
$J$-functional plays an essential role for the geometry of $ \cH$ because it is compatible with Darvas' $L^1$-distance 
\begin{equation*}
d(\phi, \psi) := \inf_{\phi^t} \int_0^1 \int_X \abs{\dot{\phi}^t} \omega_{\phi^t}^n  dt, 
\end{equation*}
where the infimum is taken over all smooth paths $\phi^t$ $(t\in [0, 1])$ connecting $\phi$ with $\psi$. See \cite{Dar15} for the detail. 

\subsection{Non-Archimedean counterparts}
G. Tian introduced the notion of K-stability for a Fano manifold, observing behavior of the K-energy along a special degeneration. 
Donaldson gave the algebraic formulation for a polarized manifold and generalizes the Futaki invariant to a test configuration where the degeneration allows a non-normal central fiber. 

\begin{dfn}
Test configuration is a family of polarized schemes $\pi\colon (\cX, \cL) \to \C$
endowed with an action $\lambda\colon \C^* \to \Aut(\cX, \cL)$ such that the general fiber is isomorphic to $(X, L)$. 
\end{dfn} 

Throughout the paper we assume $\cX$ is normal.  
We identify the fiber $(\cX_1 , \cL_1)$ with $(X, L)$ in the sequel. 
If there exists an equivariant morphism $f\colon \cX' \to \cX$ such that $f^*\cL = \cL'$ the two test configurations are considered to be equivalent. 
Replace $\cX$ with $\cX'$, we may assume that $\cX$ dominates the trivial family by $\rho\colon \cX \to \C \times X $. 
It enables us to compare $\cL$ with $\rho^*p_2^* L$. 

In \cite{BHJ16} we described the slope of the energy along an arbitrary test configuration.  
Donaldson-Futaki invariant is further interpreted into the non-Archimedan counterpart of the K-energy functional. 
For a quick introduction, we take the unique compactification $(\overline{\cX}, \overline{\cL}) \to \P^1$ such that the family outside $0 \in \P^1$ is the product space endowed with trivial $\C^*$-action. 

\begin{dfn}
Non-Archimedean Aubin-Mabuchi energy is defined by the intersection number 
\begin{equation*}
E^\NA (\cX, \cL) = \frac{\overline{\cL}^{n+1}}{(n+1)V}. 
\end{equation*}
We define non-Archimedean J-energy as 
\begin{equation*}
J^\NA (\cX, \cL) = V^{-1}\overline{\cL} (\rho^*p_2^* L)^n-E^\NA (\cX, \cL). 
\end{equation*}
Non-Archimedean K-energy is 
\begin{equation*}
M^\NA (\cX, \cL) = V^{-1}(K_{\overline{\cX}/\P^1}^{\log} \overline{\cL}^n) 
+\hat{S} E^\NA(\cX, \cL). 
\end{equation*}
\end{dfn} 

We will not use these expressions in this paper. 
When the automorphism group is not finite we still have to modify the definition of $J^\NA$ but rather use the equivalent definition in words of the induced $\C^*$-action on $H^0(\cX_0, k\cL_0)$.  
We set $N_k:=\dim H^0(\cX_0, k\cL_0)$ and denote the weights of this $\C^*$-action by $\lambda_1, \lambda_2, \dots, \lambda_{N_k}$. 
From \cite{BHJ15} subsection 7.2, it then holds that 
\begin{align*}
E^\NA (\cX, \cL) = \lim_{k \to \infty} \frac{1}{N_k}\sum_{i=1}^{N_k} \frac{\lambda_i}{k}, \ \ \ 
J^\NA (\cX, \cL) = \lim_{k \to \infty} \bigg[ \max_{1\leq i \leq N_k} \frac{\lambda_i}{k} -\frac{1}{N_k}\sum_{i=1}^{N_k} \frac{\lambda_i}{k} \bigg]. 
\end{align*}

For our purpose the most important point is that the non-Archimedean energies are actually equivalent to the slope of the original (Archimedean) energies defined on $\cH$.  
More precisely, we take a smooth fiber metric $\Phi$ of $\cL$ which gives the equivalent class of rays on $\cH$.

\begin{thm}(\cite{BHJ16})\label{BHJslope} 
Let $(\cX, \cL)$ be a test configuration and $\Phi$ a smooth fiber metric of $\cL$. 
For the associated ray $\phi^t(x):= \Phi(\lambda(e^{-t})x)$ on $\cH$ 
we have 
\begin{align*}
E^\NA (\cX, \cL) = \lim_{t \to \infty} \frac{E(\phi^t)}{t} \ \ \  \text{and}  \\ 
M^\NA (\cX, \cL) = \lim_{t \to \infty} \frac{M(\phi^t)}{t}. 
\end{align*} 
\end{thm} 

In the above it is enough to have the metric $\Phi$ defined over the unit disk $\Delta$. 
Any $\phi$ on $(X, L)=(\cX_1, \cL_1)$ extends to the boundary $\pi^{-1}(\partial \Delta)$ in the way 
$\phi(\lambda(e^{\sqrt{-1}\theta})x) =\phi(x)$ and 
there exists unique solution to the Dirichlet problem 
\begin{equation*}
\begin{cases}
(\dd \Phi)^{n+1} = 0 \ \ \ \text{on} ~\pi^{-1}(\Delta) \\
\Phi = \phi \ \ \ \text{on} ~ \pi^{-1}(\partial \Delta) 
\end{cases}
\end{equation*} 
since we assumed $\cX$ is normal. 
Moreover, the solution $\Phi$ is equivalent to the envelope of all singular fiber metric $\Psi$ which satisfies 
$$
\limsup_{z \to \zeta} \Psi(z) \leq \phi(\zeta)  
$$ 
 for any boundary point $\zeta \in \pi^{-1}(\partial \Delta)$.  
Unfortunately, however, the solution is not smooth in general. 
It is recently proved by \cite{CTW18} that $\Phi$ has optimal $C^{1, 1}$-regularity for the smooth boundary data.  
If it is smooth the associated $\phi^t$ defines a geodesic with respect to Mabuhi's Riemannian metric 
\begin{equation*}
\norm{u}^2 = \frac{1}{V}\int_X u^2 \omega_\phi^n. 
\end{equation*} 
In general $\phi^t$ obtained from the above construction is called a weak geodesic ray. 
In Theorem \ref{BHJslope} one can take the weak geodesic ray for $E^\NA$, but it seems unclear for $M^\NA$. 

\section{Slope formula for the reduced $J$-functional}

\subsection{Reduced J-functional}

From now on we fix a complex torus $T \subseteq \Aut(X, L)$ and reduce the $J$-functional in terms of $T$. 
As $\Aut(X, L)$ denotes the group of fiber-preserving endomorphisms of $L$, 
$T$ contains the tautological sub-torus $\C^*$ which acts on the fibers as the constant multiplication. 
Let $N:= \Hom(\C^*, T)$ be the lattice of one-parameter subgroups. 
Note that the compact subtorus $S:=N \otimes \S^1$ is naturally obtained. 

For any $g \in \Aut(X, L)$, we may define the pull-back of the fiber metric: $\phi_g(x) = \phi(gx)$. 
The setting is for the case when the cscK metric is known to be $S$-invariant and so that we mainly consider the $S$-invariant metrics. 
If $g \in T$ it is easy to see that $\phi_g$ is still $S$-invariant. 

Let us introduce the reduced J-energy 
\begin{equation}
J_T(\phi):= \inf_{g \in T} J(\phi_g). 
\end{equation}
This is in parallel with the reduced $L^1$-distance   
\begin{equation*}
d_T(\phi, \psi):= \inf_{g \in T} d(\phi_g, \psi) =\inf_{g \in T} d(\phi, \psi_g). 
\end{equation*}
For our purpose it is more convenient directly to use the symmetric $I$-functional 
\begin{equation*}
I(\phi, \psi) := \frac{1}{V}\int_X (\phi-\psi) \big( \omega_\phi^n-\omega_\psi^n \big), 
\end{equation*}
which is actually comparable to $J$-energy as follows: 
\begin{equation*}
\frac{1}{n+1} I(\phi, \psi) \leq J_\psi (\phi) \leq I(\phi, \psi). 
\end{equation*}
One can easily derive these inequalities, using simple integration by parts. 
By Theorem $1.8$ of \cite{BBEGZ11} $I$-functional moreover satisfies ``almost-triangle inequality" 
\begin{equation*} 
 c_nI(\phi_1, \phi_3) \leq I(\phi_1, \phi_2) +I(\phi_2, \phi_3)
\end{equation*}
so it behaves almost like a distance. The constant $c_n$ depends only on the dimension. 
In place of the reduced $L^1$-distance, we have the reduced $I$-functional 
\begin{equation}
I_T(\phi, \psi):= \inf_{g \in T} I(\phi, \psi_g) =  \inf_{g \in T} I(\phi_g, \psi). 
\end{equation}
Taking the infimum of $c_nI(\phi_1, \phi_3) \leq I(\phi_1, \phi_{2, g}) +I(\phi_{2, g}, \phi_3)$, one indeed obtains the almost-triangle inequality for $I_T$. \\

Next we discuss the non-Archimedean interpretation of the reduced J-functional. 
We assume test configurations are $T$-equivariant, in parallel with $S$-invariant metrics. 
More precisely, as the datum of the test configuration we assume the extended $T$-action on $(\cX, \cL)$,  
which also commutes with the $\C^*$-action $\lambda$. 
Then the associated weak geodesic ray $\phi^t$ emanating from given $\phi \in \cH^S$ is automatically $S$-invariant for all $t$, since it was constructed as the envelope. 
We need the following lemma, which is a consequence of the fact that a torus-equivariant vector bundle over $\C$ is equivariantly trivializable (see the proof of  \cite{Don02}, Lemma 2). 

\begin{lem}
For any $T$-equivariant test configuration $(\cX, \cL)$ 
the natural representation $T \to \Aut(\cX_0, k\cL_0)$ is faithful for any sufficiently large $k$. Moreover, the induced action $\lambda\colon \C^* \to \Aut(\cX_0, k\cL_0)$ factors through some torus $T'$ containing the image of $T$ so that it is naturally identified with the element $\lambda \in N'$. 
\qed
\end{lem} 

Since $\Aut(\cX_0, k\cL_0)$ is stable for large $k$, so is a choice of $T'$. 
Regarding the above lemma, we consider the same space $(\cX, \cL)$ endowed with a possibly irrational $\lambda \in N'_\R$ and call it {\em an irrational $T$-equivariant test configuration}. We remark that it is equivalent to allowing real parameterization for the associated finitely generated filtration (explained {\em e.g.} in \cite{BHJ15}). 

Given a $T$-equivariant test configuration $(\cX, \cL)$, one can twist it by any $\mu \in N_\R$ to obtain a new irrational  $T$-equivariant test configuration $(\cX_\mu, \cL_\mu)$. 
As a family $(\cX_\mu, \cL_\mu)$ is isomorphic to $(\cX, \cL)$ but the action is changed to $\lambda +\mu \in N'_\R$. 
If one denotes the weights for the $\C^*$-module $H^0(\cX_0, k\cL_0)$ by $\lambda_1, \cdots \lambda_{N_k}$, the same basis diagonalizes $\mu$ so that the weights for $(\cX_\mu, \cL_\mu)$ are $\lambda_1+\mu_1, \cdots,  \lambda_{N_k}+\mu_{N_k}$. 

We are going to describe $J^\NA (\cX_\mu, \cL_\mu)$ in terms of the weights. 
Notice that the compactification $\overline{\cX_\mu}$ (and hence the intersection number description of $J^\NA$) depends on the action, since it is forced to be trivial around $\infty \in \P^1$.  
We here use another description of the non-Archimedean J-functional in \cite{BHJ15}. It follows from subsection 7.2 of \cite{BHJ15} that for rational $\mu$ one has 
\begin{equation}\label{weight description} 
J^{\NA}(\cX_\mu, \cL_\mu) = \lim_{k \to \infty} \bigg[ \max_i \frac{\lambda_i+\mu_i}{k} -\frac{1}{N_k} \sum_{i=1}^{N_k}  \frac{\lambda_i+\mu_i}{k} \bigg]. 
\end{equation}
The limit  exists by the homegeneity. 
One may observe that each term in the description corresponds to the term  in the analytic definition of $J$-functional. 
As well, if the test configuration is product, \ie  if the $\C^*$-action $\lambda$ is contained in $\Aut(X, L)$, the above formula is a simple consequence of the equivariant Riemann-Roch formula. 

For general $\mu \in N_{\R}$ we extend $J^{\NA}(\cX_\mu, \cL_\mu)$ by the continuity which is a consequence of the following proposition. 
Notice that since the associated weak geodesic $\Phi$ is psh the ray $\phi^t$ is convex in $t$. One can then easily see that time derivative
\begin{equation*}
\dot{\phi}^0:=\inf_{t>0} \frac{\phi^t -\phi^0}{t}
\end{equation*}
is well-defined. 
More deep analysis shows that the weak geodesic has $C^{1, 1}$-regularity. 

\begin{prop}\label{analytic description of J^NA}
Let $(\cX, \cL)$ be a $T$-equivariant test configuration and $\phi^t$ the associated weak geodesic ray.
For any $\mu \in N_{\Q}$ we have the analytic description 
\begin{equation*}
J^{\NA}(\cX_\mu, \cL_\mu) = \sup_X (\dot{\phi}^0 +h_\mu) -\frac{1}{V}\int_X  (\dot{\phi}^0 +h_\mu) \omega_{\phi^0}^n,  
\end{equation*}
where $h_\mu$ is the normalized Hamilton function in terms of $\phi^0$. 
In particular $J^{\NA}(\cX_\mu, \cL_\mu)$ is continuous in $\mu \in N_{\Q}$. 
\end{prop}

\begin{proof}
From the main argument of \cite{His16a}, \cite{His16b} we have the convergence to the push-forward measure 
\begin{equation*}
\frac{1}{N_k}\sum \delta_{\frac{\lambda_i+\mu_i}{k}} \to (\dot{\phi}^0 +h_\mu)_*(V^{-1}\omega^n). 
\end{equation*}
On the other hand Theorem 5.16 of \cite{BHJ15} shows that whenever $\mu$ is rational 
$\max_i \frac{\lambda_i+\mu_i}{k}$ is stable for any $k$ sufficiently divisible.

\end{proof} 

\begin{dfn}
Let us define the reduced non-Archimedean $J$-functional as  
\begin{equation}
J_{T}^{\NA}(\cX, \cL) := \inf_{\mu \in N_{\R} } J^{\NA}(\cX_\mu, \cL_\mu).  
\end{equation}
\end{dfn}

The above infimum in fact attained by a rational $\mu$. 
One can further check from (\ref{weight description}) the following: 

\begin{lem}\label{PL convexity}
$J^{\NA}(\cX_\mu, \cL_\mu)$ is rational, piecewise-linear convex function in $\lambda+\mu \in N'_\R$ and properly grows in $\mu \in N_\R$.
\end{lem} 

We also remark that the $L^2$-orthogonal part of $(\cX, \cL)$ to vector fields is attained by a rational $\mu$ and J-norm of the projected test configuration is equivalent to what we defined above. 
See \cite{His16b} for the detail. 
It seems that, however, we still need the definition of $J^{\NA}(\cX_\mu, \cL_\mu)$ for general $\mu \in N_{\R}$ in proving the slope formula for the reduced $J$-functional. 

\subsection{Function of log-norm singularities}\label{log norm}

In this subsection we study $J(\phi_g)$ as a function in $g\in \Aut(X, L)$. 
If one takes $\phi$ as a Fubini-Study type fiber metric 
\begin{equation*}
\phi (x) = \frac{1}{N_k} \log \big( \abs{s_1(x)}^2 + \cdots +\abs{s_{N_r}(x)}^2 \big) 
\end{equation*}
defined by the basis 
$s_1, \dots, s_{N_r} \in H^0(X, rL)$, 
the pull-back $\phi_g$ makes sense for $g \in \GL(N_r; \C)$. 
The above exponent $r$ is taken sufficiently large and absolute value of the section regards the local trivialization.    
We define the functionals 
\begin{equation*} 
{\bf e}(g):= E(\phi_g),~ {\bf j}(g):= J(\phi_g),~ \text{and}~ \  {\bf m}(g) := M(\phi_g). 
\end{equation*}  
Let us starts from the following observation which is itself interesting.  
The author learned the argument from the communication with Professor S. Boucksom. 
Denote the identity component of the automorphism group by $\Aut^0(X, L)$. 

\begin{thm}\label{plurisubharmonicity} 
The energies ${\bf e}$ and ${\bf m}$ are pluriharmonic and ${\bf j}$ is plurisubharmonic along $\Aut^0(X, L)$. 
\end{thm} 

\begin{proof}
Let us take an arbitrary holomorphic map $g \colon \Delta \to \Aut^0(X, L)$ which sends $z\in\Delta$ in the one dimensional disk to $g(z)$. We have the well-known fiber integration formula:  
\begin{equation*}
\dd_z E(\phi_{g(z)}) = (n+1)^{-1}V^{-1} \int_X (\dd_{z, x} \phi_{g(z)}(x))^{n+1}. 
\end{equation*}
Define the holomorphic map $F\colon \Delta\times X \to X$ by $F(z, x):=g(z)x$ and proceed as  
\begin{equation*}
(\dd_{z, x} \phi_{g(z)}(x))^{n+1} 
= (\dd_{z, x} F^*\phi)^{n+1} 
=F^* (\dd_{x} \phi)^{n+1} =0.  
\end{equation*}
This completes the proof for ${\bf e}$. Plurisubharmonicity of ${\bf j}$ is immediate. Let us focus on the {\bf m}. We regard K-energy as a metric on the Deligne pairing $\langle K_{\Delta \times X /\Delta}\cL^n\rangle +(n+1)^{-1} \hat{S} \langle \cL^{n+1} \rangle $ with $\cL := F^* L $ (see {\em e.g.} \cite{BHJ16} for such identification). Then the proof is reduced to show 
\begin{equation*}
\dd_{z, x} \big\langle \log (\dd\phi_{g(z)})^n, \phi_{g(z)}, \dots, \phi_{g(z)} \big\rangle =0,  
\end{equation*}
where the Monge-Amp\'ere operator is taken over $X$. 
Endow the metric $ \Psi := \log (dd^c\phi_{g(z)})^n$ on the relative canonical $K_{\Delta \times X /\Delta}$ and  $\psi:= \log (dd^c\phi)^n$ on $K_X$. Since the natural automorphism $G\colon (z, x) \mapsto (z, g(z)x)$ of $\Delta \times X$ induces $F^*K_X \simeq p_2^*K_X =  K_{\Delta \times X /\Delta}$ it is enough to show 
\begin{equation*}
F^* \psi = \Psi. 
\end{equation*}
This can be checked fiberwisely. Indeed for each fixed $z$ we have the inclusion $i_z\colon X \to \Delta \times X $ such that 
\begin{align*}
e^{\Psi_z} &= (dd^ci_z^* F^*\phi)^n = i_z^* (dd^c_{z, x}F^*\phi)^n \\
 &= i_z^*F^* (dd^c\phi)^n = g(z)^*  (dd^c\phi)^n. 
\end{align*}
\end{proof} 

Plurisubharmonicity on $\GL(N_r, \C)$ is not the case. 
A fundamental result of \cite{Paul12} (see also \cite{Li12}) tells us that asymptotic of these functions are controlled by the difference of two plurisubharmic functions. 
More specifically it is described by the actions to two homogeneous polynomials; $X$-resultant $R$ and $X$-hyperdiscriminant $\Delta$ associated with the Kodaira embedding. 
To be precise, taking norms $\norm{\cdot}$ of $\GL(N_r;\C)$-vector space and the Hilbert-Schmidt norm $\norm{\cdot}_{\HS}$ of $\GL(N_r;\C)$ itself we have 
\begin{align*}
&{\bf m}(g) = V^{-1} \log\norm{g\cdot \Delta} - V^{-1} \frac{\deg \Delta}{\deg R}\log\norm{g\cdot R} +O(1) \ \ \text{and} \\ 
&{\bf j}(g) =  V^{-1}\log \norm{g}_{\HS} - V^{-1} \frac{1}{\deg R}\log\norm{g\cdot R} +O(1). 
\end{align*}
Moreover the second term for ${\bf j}$ just corresponds to the Aubin-Mabuchi functional. 
The point here is that we can not expect to have a single log-norm term as we had in the classical GIT setting. 
In the terminology of \cite{BHJ16}, for a general reductive group $G$ a function $f\colon G\to \R$ of the form 
\begin{equation*}
f(g) = a\log \norm{g\cdot v} - b\log \norm{g \cdot w} +O(1)
\end{equation*}
is said to have {\em log-norm singularities}. Even in this generalized ``pair of log-norm terms " setting, we have the correspondence of  the Hilbert-Mumford weight for a given one-parameter subgroup. 

\begin{thm}[Theorem $4.4$, (i) of  \cite{BHJ16}]\label{Hilbert-Mumford}
Let $f$ be a function on $G$ with log norm singularities. Then, 
for each one-parameter subgroup $\lambda\colon \C^* \to G$ there exists $f^{\NA}(\lambda) \in \Q$ such that 
\begin{equation*}
f(\lambda(\tau)) = f^{\NA}(\lambda) \log \abs{\tau}^{-1} +O(1) 
\end{equation*} 
for $\abs{\tau}\leq1$. 
\end{thm}

It is crucial how $O(1)$ term depends on $\lambda$ so we sketch the proof. 

\begin{proof}
Assume that the image of $\lambda $ is contained in the fixed torus $T \subseteq G$. 
Let 
\begin{equation*}
V=\sum_{\mu \in M} V_\mu
\end{equation*}
be the weight decomposition for $T$ so that 
$\lambda(\tau) v_\mu= \tau^{\langle \mu, \lambda \rangle} v_\mu$ if $v_\mu \in V_\mu$. 
Writing $v=\sum_\mu v_\mu$, we obtain
\begin{equation*}
\log \norm{\lambda(\tau)\cdot v}
=\max_{v_\mu \neq 0} \big(\langle \mu, \lambda \rangle \log \abs{\tau} +\log \norm{v_\mu} \big) +O(1)
=-\big(\min_{v_\mu\neq 0}\langle \mu, \lambda \rangle \big) \log \abs{\tau}^{-1} +O(1)
\end{equation*}
for $\abs{\tau} \to 0$. We may take  
$ f^\NA(\lambda) :=-\min_{v_\mu\neq 0}\langle \mu, \lambda \rangle $. 
\end{proof} 

\begin{rem}\label{O(1)}
In the proof of Theorem B, we will use the above theorem when $G=T \subseteq \Aut(X, L)$. 
The above proof shows that one can take $O(1)$ term uniform in $\lambda \in N$. 
The uniformity fails if the image of $\lambda\colon \C^* \to G$ is not contained in the fixed torus. 
\end{rem} 

\subsection{Proof of Theorem B}

Let us take a $T$-equivariant test configuration $(\cX, \cL)$ to show the slope formula: 
\begin{equation*}
J_{T}^{\NA}(\cX, \cL) = \lim_{t \to \infty} J_T(\phi^t)/t. 
\end{equation*} 

We first observe that the limit on the right-hand side exists. Indeed a function $F(g, \tau):=J(\phi_g^{-\log \abs{\tau}})$ on $T \times \Delta$ is plurisubharmonic by Theorem \ref{plurisubharmonicity}. It is moreover $S$-invariant since $\Phi$ is $S$-invariant. We conclude by Kiselman's minimum principle that 
\begin{equation*}
\inf_{g \in T} F(g, \tau) 
\end{equation*} 
is pluriharmonic in $\tau \in \Delta$, hence $J_T(\phi^t)$ is convex. 
Combined with the linear bound $J_T(\phi^t) \leq J(\phi^t) \leq Ct$ convexity assures the limit. 

Notice that any choice of $\Phi$ gives the same limit. 
For any other $\Psi$, difference of two fiber metrics $\Phi-\Psi$ is uniformly bounded over the unit closed disc. 
It is easy to check the inequality 
\begin{equation*}
\abs{J_T(\phi^t)-J_T(\psi^t)} \leq 2\sup_X \abs{\phi^t - \psi^t}. 
\end{equation*}

Next we show 
\begin{lem}
For any $\phi$ there exists $g \in T$ which attains the infimum  $J_T(\phi)=J(\phi_g)$. 
\end{lem} 
\begin{proof}
From Theorem \ref{Hilbert-Mumford} 
the function ${\bf j}(g)=J(\phi_g)$ is controlled as 
\begin{equation*}
{\bf j}(\mu(e^{-t})) = J^\NA (\mu)t + O(1),  
\end{equation*}
where $J^\NA (\mu)$ is equivalent to the non-Archimedean functional of the product test configuration generated by $\mu$. 
As we saw in Remark \ref{O(1)} the above $O(1)$ term is independent of $\mu \in N$. We obtain the same conclusion for $\mu \in N_\Q$ by the homogeneity. 
From the description (\ref{weight description}) for the trivial test configuration it is easy to see that $J^\NA \geq 0 $ for $\mu \in N_\Q$ and the equality holds iff $\mu$ is trivial. The homogeneity and the continuity property extends $J^{\NA}$ to a rational, piecewise-linear strictly convex function on $N_\R$. We claim that ${\bf j}(g) \to \infty$ as $ g\to \infty$ in $T$. If not, we have a sequence $g_k \to \infty $ with ${\bf j}(g_k)$ bounded. Denote the first projection of $g_k \in T=N_\R \times S$ by $\log \abs{g_k}$. 
Note that even if $\log \abs{g_k} \in N_\R$ is irrational $\log \abs{g_k} (e^{-t}) \in T$ is well-defined for any $t \in [0, \infty)$.  
Since $\phi$ is $S$-invariant it yields 
\begin{equation*}
{\bf j}(g_k) ={\bf j}(\log \abs{g_k}(e^{-1})) 
=J^\NA (\log \abs{g_k}) + O(1)
\end{equation*}
which is a contradiction. 
\end{proof}

Now we prove the slope formula. 
When $\mu \in N$, the twisted ray $\phi^t_{\mu(e^{-t})}$ is associated with $(\cX_\mu, \cL_\mu)$.   
It follows 
\begin{equation*}
J^{\NA}(\cX_\mu, \cL_\mu) = \lim_{t \to \infty} J(\phi^t_{(\mu)(e^{-t})})/t
\geq \lim_{t \to \infty} J_T(\phi^t)/t.   
\end{equation*}
The homogeneity shows the same conclusion for $\mu \in N_\Q$. 

The problem is another direction. 
From the lemma we may take $g_t$ so that 
$J(\phi^t_{g_t})=J_T(\phi^t)$. 
The almost triangle inequality shows 
\begin{equation*}
c_nI(\phi^0, \phi^0_{g_t^{-1}})
\leq I(\phi^0, \phi^t) + I(\phi^t, \phi^0_{g_t^{-1}}) 
= I(\phi^0, \phi^t) + I(\phi^0, \phi^t_{g_t}). 
\end{equation*}
Since the last term $I(\phi^0, \phi^t_{\s_t})$ is bounded from above by 
\begin{equation*}
 (n+1)J(\phi^t_{\s_t})=(n+1)J_T(\phi^t)\leq (n+1)J(\phi^t) \leq Ct, 
\end{equation*}
we obtain the linear bound
\begin{equation*}
{\bf j}(g_t)=J(\phi^0_{g_t}) \leq Ct.  
\end{equation*}
The left-hand side is equivalent to 
\begin{equation*}
{\bf j}(\frac{1}{t}\log \abs{g_t}(e^{-t})) = J^{\NA}(\frac{1}{t}\log \abs{g_t})t +O(1) 
\end{equation*}
hence it implies 
\begin{equation*}
J^{\NA}(\frac{1}{t}\log \abs{g_t}) \leq C. 
\end{equation*} 
As we already observed that $J^{\NA}$ is a piecewise-linear strictly convex function on $N_{\R}$ the above bound produces a convergent subsequence $\frac{1}{t_k}\log \abs{g_{t_k}} \to \mu \in N_\R$. The uniform convergence 
\begin{equation*}
\Phi((\lambda + \frac{1}{t_k}\log \abs{g_{t_k}})(e^{-t_k})z)
\to \Phi((\lambda +\mu)(e^{-t_k})z) 
\end{equation*}
follows from the compactness of $\cX$ over the unit closed disk. With the estimate 
\begin{equation*}
\abs{J(\phi^t_{g_t}) - J(\phi^t_{\mu(e^{-t})})}
\leq 2 \sup_X \abs{\phi^t_{g_t}-\phi^t_{\mu(e^{-t})}}, 
\end{equation*}
we conclude 
\begin{equation*}
\lim_{t \to \infty} J_T(\phi^t)/t = \lim_{t \to \infty} J(\phi^t_{\mu(e^{^{-t}})})/t 
=J^{\NA}(\cX_\mu, \cL_\mu) \geq J^{\NA}_T(\cX, \cL). 
\end{equation*}


\section{Coercivity for the toric case}

In this section we explain how the the reduced J-functional and stability is translated into the toric terminology. 
We start from quickly reviewing the toric setting for the readers' convenience. See \cite{Gui94}, \cite{Abr98}, and \cite{Don02} for the detail. 
Let $M$ be a integral lattice of rank $n$. A toric polarized manifold is defined by a Delzant polytope $P$ in $M_{\R}:=M\otimes_{\Z} \R$. In our convention $P$ contains its boundary and we distinguish the interior as $P^\circ$. Notation for the dual lattice $N$ consists with the previous sections. The complex torus $T:= N \otimes \C^*$ naturally acts on $X$ with the open dense orbit.

\subsection{Torus-invariant metrics}
Every moment map $\mu\colon X \to M_{\R}$ gives a homeomorphism $X_0 :=\mu^{-1}(P^\circ) \simeq S \times P^\circ$ so that $N_{\R} \times P^{\circ}$ universally  covers $X_0$. 
We further take a coordinate $(x_1, \dots, x_n)$ of $M_{\R}$. 
Dual coordinate $(\eta_1, \dots, \eta_n)$ of $N_{\R}$  induces the complex coordinate $\log z_i = \xi_i + \i \eta_i$ of $X_0\simeq (\C^*)^n$. The K\"ahler metric defining $\mu$ is written in $X_0$ as 
\begin{equation*}
\omega = \sum dx_i \wedge d \eta_i. 
\end{equation*}
Conversely, any $S$-invariant K\"ahler metric $\omega_\phi= \dd\phi$ is represented by the local function $\phi$ on $X_0$. If one denotes it by $\phi (\xi_1, \dots \xi_n)$ the gradient 
\begin{equation*}
(z_1, \dotsm z_n) \mapsto \big( \frac{\partial\phi}{\partial \xi_1}, \dots,  \frac{\partial\phi}{\partial \xi_n}\big)
\end{equation*} 
gives the moment map for $\omega_\phi$. Notice that the image $P$ is independent of these additional data. 

The Legendre transform 
\begin{align*}
u(x_1, \dots, x_n) &:= \sup \big\{ \sum x_i \xi_i -\phi(\xi) \big\} \\
&= \sum x_i \xi_i -\phi(\xi) \ \  \text{if} \ x_i=\frac{\partial\phi}{\partial\xi_i}
\end{align*} 
is called toric potential. By the standard Delzant construction, one can prove that $u$ defines a convex function on $P^\circ$. 
Moreover, $u$ is characterized by the Guillemin boundary condition. 
That is, if $P$ is written in the form 
\begin{equation}
P= \big\{ x\in M_\R:  \ \ell_k(x):= \langle x, \a_k \rangle -\b_k  \geq 0  \ \text{for any} \  1\leq k\leq d \big\}, 
\end{equation}
then $u- \frac{1}{2} \sum_{k=1}^d \ell_k \log \ell_k$ is smooth up to the boundary. Conversely, such a convex function produces an $S$-invariant K\"ahler metric. 
We denote by $\cS$ the collection of toric potentials with the Guillemin boundary condition. 
We also need the space $\cC_{\infty}$ which consists of convex functions continuous on $P$ and smooth in the interior. 

If one take an affine function and rescales $u$ to $u- \sum a_i x_i -b$, $\phi$ changes into $\phi(\xi_i+a_i)+b$. 
Note that the change of variables $\xi_i \mapsto \xi_i +a_i$ is just given by the torus-action. 
We may rescale $u$ to have a point $x_0 \in P^{\circ}$ as a minimizer.
Rescaling again by  
\begin{equation*}
 \sum a_i x_i +b = \sum \frac{\partial u}{\partial x_i}(x^0)(x_i -x^0_i) +u(x^0)
\end{equation*} 
one obtains the toric potential $u$ satisfying 
\begin{align*}
&(i) \ \ \inf_P u =u(x_0) =  0  \ \ \text{and} \\
&(ii) \ \ \frac{\partial u}{\partial x_i}(x_0) =0 \ \ \text{ for all} \ 1\leq i\leq n. 
\end{align*} 
We say $u$ is normalized at this end. 

The scalar curvature is given by a fourth-order differential of $u$. 

\begin{thm}[\cite{Abr98}]\label{Abreu}
For any toric potential 
\begin{itemize} 
\item[$(i)$]
Hessian is transformed as 
\begin{equation*}
(u_{ij}):=\big( \frac{\partial^2 u}{\partial x_i \partial x_j} \big) = \big( \frac{\partial^2\psi}{\partial\xi_i\partial \xi_j} \big)^{-1} \ \ \text{if} \ x_i=\frac{\partial\psi}{\partial\xi_i}. 
\end{equation*} 
\item[$(ii)$]
Let $(u^{ij}):=(u_{ij})^{-1}$. The scalar curvature is given by fsymplectic
\begin{align*}
S_{\phi}
&= -\frac{1}{2} \sum_{i, j=1}^{n} \frac{\partial^2 u^{ij}}{\partial x_i \partial x_j} \ \ \text{if} \ x_i=\frac{\partial\psi}{\partial\xi_i}. 
\end{align*} 
\end{itemize} 
\end{thm}
\qed

Based on the above \cite{Don02} proved the following formula.  

\begin{thm}\label{toricK-energy} 
Let $u$ be the toric potential of an $S$-invariant metric $\phi$. Then the energy functionals are written as 
\begin{align*}
& E(\phi) = -\frac{1}{V} \int_P u \ \ \text{and} \\
& M(\phi) = -\frac{1}{2}\int_P \log \det (u_{ij}) + \int_{\partial P} u -  \hat{S} \int_Pu. 
\end{align*}
\end{thm} 
\qed

It is remarkable in the toric case that the linear part is given by 
\begin{equation*}
L(u) = \int_{\partial P} u - \hat{S} \int_P u. 
\end{equation*}
We will consider the energy as a function in $u$ and write as $E(u)$ or $M(u)$ in abuse of notation. 
For $J$-functional we recall: 
\begin{lem}[\cite{ZZ08b}, Lemma $2.1$]\label{toricJ}
There exists a constant $C>0$ such that for any normalized toric potential $u$ 
\begin{equation*}
\abs{\ J(\phi) -\frac{1}{V}\int_P u \ } \leq C. 
\end{equation*} 
\end{lem}
\qed


\subsection{Toric test configurations}

Let us continuously fix the coordinate of $M_{\R}$. 
By definition of the moment map, any Hamilton function associated with $\mu \in N_{\R}$ is translated into the affine function on $P$. General f configuration corresponds to a convex, rational piecewise-linear function $f\colon P\to \R$. In fact the compactification $(\overline{\cX}, \overline{\cL}) \to \P^1$ is constructed by the $(n+1)$-dimensional polytope 
\begin{equation}\label{bigpolytope}
\fP : = \big\{ (x, y) \in M_\R\times\R: \ f(x) \leq y \leq B \big\}, 
\end{equation}
where the constant $ B \geq \max_P f$ is taken sufficiently large. 
Note that $\cX$ is necessary normal.  
Non-Archimedean K-energy precisely equals to 
\begin{equation}
L(f) =  \int_{\partial P} f - \hat{S}\int_P f.   
\end{equation} 
We have a combinatorial description of the volume $V= \vol(P)$ and 
\begin{equation}
\hat{S}= \frac{\area(\partial P)}{\vol(P)}. 
\end{equation}
The above integration on $P$ is for the Lebesgue measure of $M_{\R}\simeq \R^n$ and the area measure $d\sigma$ on the facet $\{ x\in P: \ell(x)=0\} $ is determined by 
\begin{equation*}
 dx_1\cdots dx_n = \pm d\sigma \wedge d\ell. 
\end{equation*}

\begin{rem}
The above $L(f)$ is not completely equivalent to the Donaldson-Futaki invariant introduced by \cite{Don02}. Indeed $L(f)$ should be homogeneous for the {\em normalized} finite base change while the Donaldson-Futaki invariant is not. 
They are equivalent when the central fiber $\cX_0$ is reduced. See \cite{BHJ15} for the detail. 
\end{rem}

We will measure the positivity of $L(f)$, in terms of the new J-norm 

\begin{equation*}
\norm{f}_J := \inf_{\ell}\bigg\{  \frac{1}{\vol(P)}\int_P (f+\ell) -\min_P \{ f+\ell\} \bigg\}, 
\end{equation*}

where $\ell$ runs through all affine functions. 
It is easy to see that $\norm{f}_J \geq 0$ and the equality holds if and only if $f$ is an affine function. 
In the toric setting, this can be replaced to the $L^1$-norm once the convex functions are normalized. 
Recall that a convex function is said to be normalized if $f \geq 0$ and $f(0)=0$. 

\begin{lem}
There exists a constant $\d >0$ such that 
\begin{equation*}
\frac{\d}{V} \int_P f \leq \norm{f}_J \leq \frac{1}{V}\int_P f 
\end{equation*} 
holds for any normalized convex function $f$. 
\end{lem} 

\begin{proof}
It is clear $V\norm{f}_J \leq \int_P f $. 
Suppose $\d \int_P f \leq V\norm{f}_J$ is not true. 
From a standard approximation argument, there exists a sequence of normalized $f_k \in \cC_\infty$ such that 
$\int_P f_k=1$
and $\norm{f_k}_J \to 0$ as $k \to \infty$. 
By the compactness result (Corollaly $5.2.5$ in \cite{Don02}) there exists a subseqence (still denoted by $f_k$) which converges locally uniformly to a convex function $f \geq 0$ defined on $P^\circ$. 
It forces $\norm{f}_J=0$ and contradicts to $\int_P f_k=1$. 
\end{proof}

Finally we mention to the relation with the Donaldson's boundary norm 

\begin{equation}
\norm{f}_b := \int_{\partial P} f
\end{equation}
defined for any normalized convex function $f$, continuous up to the boundary. 
The definition is valid only for the toric situation, however, it is powerful to simplify the arguments. 
We introduce the equivalent definitions of toric uniform K-stability by the following proposition. 

\begin{prop}\label{norm comparison}
For any toric polarization, the following conditions are equivalent. 
\begin{itemize}
\item[$(1)$]
There exists $\e>0$ such that 
$L(f) \geq \norm{f}_b$ for any normalized convex, rational piecewise-linear function $f\colon P \to \R$.  
\item[$(2)$]
There exists $\e>0$ such that 
$L(f) \geq \e \int_P f$ for any normalized convex, rational piecewise-linear function $f\colon P \to \R$.  
\item[$(3)$]
There exists $\e>0$ such that 
$L(f) \geq \e \norm{f}_J $ for any convex, rational piecewise-linear function $f\colon P \to \R$. 
\end{itemize} 
\end{prop} 

\begin{proof}
Using the polar coordinate, one may easily see that $\frac{1}{V}\int_P f \leq C\int_{\partial P}f$ holds whenever $f$ is normalized. 
Therefore $(1)\Rightarrow(2)\Rightarrow(3)$ is clear. 

Suppose $(1)$ is not true. 
There exists a sequence of normalized $f_k \in \cC_\infty$ such that 
\begin{equation}\label{boundary norms are bounded}
\int_{\partial P} f_k =1
\end{equation}
and $L(f_k) \to 0$ as $k \to \infty$. 
From (\ref{boundary norms are bounded}), by the compactness result (Proposition $5.2.6$ in \cite{Don02}) there exists a subseqence (still denoted by $f_k$) which converges locally uniformly to a convex function $f \geq 0$ defined on $P^\circ$ and indeed continuously extended to any facet of $P$. 
If we assume $(3)$, 

\begin{equation*}
\e \d V^{-1}\int_P f_k \leq \e \norm{f_k}_J \leq L(f_k) \to 0
\end{equation*}

implis $\int_P f=0$ and hence $f \equiv 0$ on $P$. 
On the other hand, 
\begin{equation*}
L(f_k) = \int_{\partial P} f_k -\hat{S}\int_{P} f_k 
= 1 - \hat{S} \int_{P} f_k \geq \frac{1}{2} 
\end{equation*} 
for any sufficiently large $k$. which is a contradiction.  
\end{proof}







\subsection{Proof of Theorem D}

Theorem B shows that coercivity implies the stability in the toric sense. 
In the toric setting given a test configuration $f$ and the reference potential $u$, we may regard $u+tf$ as the dual description of the associated geodesic.  
The reader may refer to \cite{SZ12} for one exposition to this idea. 
Notice however that for the reduced norm we face the same problem; the affine function $\ell_t$ which attains $J_T(u+tf)=J(u+tf+\ell_t)$, depends on $t$. 

Let us now prove the rest half of Theorem D; stability to coercivity direction. 
By a standard approximation result (Proposition 5.2.8 and Corollary 5.2.5 of \cite{Don02}), we may assume that $L(u) \geq \e \norm{u}_J$ holds for any $u\in \cS$: toric potential with the Guillemin boundary condition. 
. 

Fix $u_0 \in \cS$ to set 
\begin{align*}
&L_0(u):= \int_{\partial P} u  -\int_P  \bigg(-\frac{1}{2} \sum_{i, j} \frac{\partial^2 u_0^{ij}}{\partial x_i \partial x_j} \bigg)u \ \ \text{and} \\ 
&M_0(u):= -\frac{1}{2}\int_P \log \det(u_{ij}) + L_0(u).   
\end{align*}
The idea of the proof originates from the following lemma. 
\begin{prop}[\cite{Don02}, Proposition $3.3.4$]\label{minimizer}
For any $u \in \cC_{\infty}$, $M_0(u) \geq M_0(u_0)$ holds. 
\end{prop} 
\qed

If $u\geq 0$, H\"older's inequality shows 
\begin{equation*}
\abs{L(u)-L_0(u)} = \abs{ \ \int_P \bigg( \hat{S} + \frac{1}{2} \sum_{i, j} \frac{\partial^2 u_0^{ij}}{\partial x_i \partial x_j} \bigg) u \ } \leq \frac{C}{V}\int_P u.  
\end{equation*}  
For an arbitrary $u \in \cS$ we take an affine function $\ell$ and apply the above inequality to $(u+\ell) - \min_P \{ u+\ell \} \geq 0$. 
It follows 
\begin{equation*}
\abs{L(u)-L_0(u)} \leq C \norm{u}_J. 
\end{equation*} 
Next we take a large $k \in \N$ and decompose the J-norm as 
\begin{equation*}
\norm{u}_J = (k+1)\norm{u}_J  -k\norm{u}_J. 
\end{equation*}
The first term $(k+1)\norm{u}_J$ is estimated by the stability $L(u) \geq \e \norm{u}_J$ so that 
\begin{equation*}
\abs{L(u)-L_0(u)}  \leq C(k+1)\e^{-1}L(u) -  C k \norm{u}_J. 
\end{equation*} 
Therefore K-energy is bounded from below by 
\begin{align*}
M(u) &\geq -\frac{1}{2}\int_P \log \det(u_{ij}) + \frac{1}{1+C(k+1)\e^{-1}} L_0(u) + \frac{Ck}{1+C(k+1)\e^{-1}}\norm{u}_J\\
&= M_0(\frac{u}{1+C(k+1)\e^{-1}}) - n \log (1+C(k+1)\e^{-1}) +  \frac{Ck}{1+C(k+1)\e^{-1}}\norm{u}_J.  
\end{align*}
The first term $M_0$ is bounded from below by Proposition \ref{minimizer}. 
This is the point where we need $\cC_\infty$ because the constant mutiple of $u$ does not satisfy the Guillemin boundary condition. 
As $k \to \infty$ the coefficient of $J$-norm approaches to $\e$. 
If $u$ is normalized Lemma \ref{toricJ} yields 
$\norm{u}_J \geq \d V^{-1}\int_P u \geq \d J(u)-C$. 
Summarizing up  we obtain constants $C'$ and $\e'$ such that 
\begin{equation*}
M(u) \geq \e' J(u) -C' 
\end{equation*} 
holds for any normalized $u \in \cS$. 

Let us finally rescale $u$. For any affine function $\ell$ we have
\begin{equation}
M(u +\ell ) = M(u)+L(\ell). 
\end{equation} 
One can easily check this by Theorem \ref{toricK-energy}.  
The stability assumption yields $L(\ell)\geq \norm{\ell}_J=0$. 
Since $-\ell$ is also affine we have $L(-\ell)\geq0$ . Thus $L(\ell)=0$ and we conclude 
\begin{equation*}
M(u) \geq \e' \inf_{\ell} J(u+ \ell) -C' 
\end{equation*} 
for any $u \in \cS$. 
As we noted, addition of an affine function corresponds to the torus action. 

\qed

\begin{rem}\label{constant}
The last part of the proof follows from the pluriharmonicity of ${\bf m}(g)$ along $\Aut^0(X, L)$. 
Actually $T$-coercivity implies ${\bf m}(g)$ bounded from below on the quasi-projective variety hence it is constant. 
Combined with the slope formula we have $M^{\NA}(\cX, \cL)=0$ for  (equivariant) product configurations. 
\end{rem}

If one require the a priori stronger coercivity condition in terms of the boundary norm: 
\begin{equation}\label{boundary coercivity}
M(u) \geq \e \norm{u}_b -C
\end{equation} 
for all normalized toric potential $u$, it simply implies the stability against the boundary norm. 
Indeed for a reference normalized $u \in \cS$ and normalized $f \in \cC_\infty$ we observe 
\begin{align*}
L(tf) 
&= \frac{1}{2}\int_P \log \frac{\det((u+tf)_{ij})}{\det(u_{ij})} +M(u+tf) -M(u) \\
&\geq \e \norm{u+tf}_b -C_u. 
\end{align*}
Dividing the both sides by $t$ and letting $t \to \infty$, we conclude $L(f) \geq \e \norm{f}_b$. 
Conversely, one can show that the stability $L(f) \geq \e \norm{f}_b$ implies (\ref{boundary coercivity}), by the same argument as in the proof of Theorem D. 
Since we have Lemma \ref{toricJ}, the condition (\ref{boundary coercivity}) clearly implies our original definition of the coercivity $M \geq \e J_T -C$. 
By Theorem D and Proposition \ref{norm comparison}, the two coercivity conditions are after all equivalent, however, it seems more subtle to derive this fact without using Theorem D. 


\section{Remarks about stability and coercivity in the general setting}

In relation with Theorem B, we discuss the definition of stability proposed in the introduction, which would work for general polarizations and automorphism groups. 

\begin{dfn}
Let $G\subseteq \Aut^0(X, L)$ be a reductive subgroup and $C(G)$ the center of $G$. Take a compact subgroup $K$ such that $G=K_\C$. 

We say that K-energy functional is $G$-coercive  if it satisfies the growth condition 
\begin{equation*}
M(\phi) \geq \e  J_{C(G)}(\phi) -C 
\end{equation*}
for any $K$-invariant metric $\phi \in \cH^K$. 

We say a polarized manifold $(X, L)$ is uniformly K-stable for $G$ if there exists a constant $\e>0$ such that 
\begin{equation*}
M^\NA (\cX, \cL) \geq \e J_{C(G)}^{\NA}(\cX, \cL)
\end{equation*} 
holds for any $G$-equivariant test configuration $(\cX, \cL)$. 
\end{dfn} 

Notice that $\phi_g$ is $K$-invariant for any $g \in C(G)$. 
Theorem B immediately implies:  
\begin{thm}
If K-energy functional is G-coercive, then the polarized manifold is uniformly K-stable for $G$. 
\end{thm}

If $X$ is a Fano manifold we have the natural polarization $L=-K_X$. 
In this case we have another energy characterized by the differential 
\begin{equation*}
\d D(\phi) =  \frac{1}{V}\int_X ( \d \phi) (e^\rho-1) \omega_\phi^n,  
\end{equation*}
where $\rho$ is the normalized Ricci potential of $\omega_\phi$. 
A critical point gives the K\"ahler-Einstein metric which is equivalent to the constant scalar curvature metric. 
Non-Archimedean D-energy $D^\NA(\cX, \cL)$ for a test configuration is defined and it gives the slope of the energy. 
See \cite{Berm16}, \cite{BHJ16} for the detail.  
We say a Fano manifold $X$ is uniformly D-stable for $G$ if there exists a constant $\e>0$ such that 
\begin{equation*}
D^\NA (\cX, \cL) \geq \e J_{C(G)}^{\NA}(\cX, \cL)
\end{equation*} 
holds for any $G$-equivariant test configuration $(\cX, \cL)$ of $(X, -K_X)$. 

Our formulation for the uniform K-stability can also be adopted to D-stability. 
In particular G-coercivity of D-energy implies uniform D-stability for $G$. 
Moreover, following \cite{BBJ15} we may derive G-coercivity of D-energy from the parallel uniform stability. 
This issue will be developed in our subsequent paper.

For a general polarization it seems still difficult to derive coercivity from uniform stability. We proceed to explain the relation with constant scalar curvature K\"ahler metrics. The next statement shows that we should take $G$ as the full automorphism group. 

\begin{thm}\label{metric to coercivity}
If a polarized manifold $(X, L)$ admits a constant scalar curvature K\"ahler metric, then 
$\Aut^0(X, L)$ is reductive and 
K-energy functional is coercive for $G=\Aut^0(X, L)$.  
We take $K$ as the group of isometry. 
For a K\"ahler-Einstein Fano manifold, D-energy is coercive for $\Aut^0(X, L)$. 
\end{thm} 
\begin{proof}
This follows from the proof of Theorem $7.2$ in \cite{DR17}, as soon as one has the regularity result of \cite{BDL16}. 
Notice that if there exists a cscK metric then $\Aut^0(X, L)=K_\C$ by the theorem of Matsushima-Lichnerowicz. 
We apply Theorem $3.4$ of \cite{DR17} to the space of $K$-invariant metrics endowed with the action of the center $C(\Aut^0(X, L))$. 
In the same way as \cite{DR17}, it is sufficient to check the condition (P3) and (P5) of their Hypothesis $3.2$.
If there exists a cscK metric, Theorem $1.4$ of \cite{BDL16} states that any minimizer of K-energy (extended to singular metrics) is automatically smooth and hence cscK. 
It assures the condition (P3). It follows from \cite{BB14} that the two cscK metrics are connected by some $g \in \Aut^0(X, L)$. 
We may moreover claim $g \in C_K(\Aut^0(X, L))$ using an elementary group theory (see Theorem $7.7$ and the proof of Claim $7.9$ in \cite{DR17}). Notice that $C_K(\Aut^0(X, L))=C(\Aut^0(X, L))$, since $\Aut^0(X, L)=K_\C$. Thus we have the condition (P5). 
Proof for D-energy is totally the same.  
\end{proof} 

For a Fano manifold the coercivity in fact implies the existence of a metric. 
The argument shows that the center actually controls the whole group. 

\begin{thm}\label{coercivity to metric in the Fano case} 
Let $X$ be a Fano manifold. Assume that D-energy functional is coercive for a reductive subgroup $G=K_\C \subseteq \Aut(X, -K_X)$. Then $X$ admits a $K$-invariant K\"ahler-Einstein metric in the first Chern class. 
\end{thm}
\begin{proof}
The proof relies on the variational approach developed by a seminal work of \cite{BB10}, \cite{BBGZ13} and \cite{BBEGZ11}. 
First we note that $J_{C(G)}$ is an exhaustion function on $\cH^K/C(G)$. 
One can directly check it using the argument for $J$ in \cite{BBGZ13}, or it also follows from the fact that $J_{C(G)}^{\NA}(\cX, \cL)$ is strictly positive for any non-product equivariant test configuration (see \cite{BHJ15} for $C(G) =\{\id\}$ case). 
The coercivity condition
\begin{equation*}
D(\phi) \geq \e J_{C(G)}(\phi) -C  
\end{equation*}
on $\cH^K$ implies that Ding functional has a minimizer $\phi$ on the $L^1$-completion $\cE^{1, K}/C(G)$. 

It is easy to check that the associated function ${\bf d}(g)= D(\phi_g)$ in $g \in \Aut^0(X, L)$ is pluriharmonic, similarly to Theorem \ref{plurisubharmonicity}. 
The same argument as Remark \ref{constant} yields that ${\bf d}(g)$ is constant along the center. 

Now we recall that for any one-parameter subgroup $\mu\colon \C^* \to G$ the slope of ${\bf d}(\mu(e^{-t}))$ is equivalent to the classical Futaki character. 
Since the character is defined on the reductive Lie algebra $\fg$ (written as the direct sum of the center and the derived algebra) the slopes are nontrivial only along the center. Therefore ${\bf d}(g)$ is actually constant along $G$. 
The differentiability result of \cite{BB10} assures 
\begin{equation*}
\int_X u (d D)_\phi =0 
\end{equation*}
for any $K$-invariant smooth function $u$. Since ${\bf d}(g)$ is constant along the whole $G$, the measure $\mu:= (d D)_\phi$ is $K$-invariant. It then follows that for any smooth function $v$ and $g \in K$ 
\begin{equation*}
\int_X v\mu =\int_X g_*(v\mu)
=\int_X ((g^{-1})^*v)g_*(\mu) 
=\int_X ((g^{-1})^*v)\mu.  
\end{equation*}
Integrating against the Haar measure we have 
\begin{equation*}
\int_X v \mu = \int_X u\mu =0. 
\end{equation*}
It yields $\mu=0$ as desired. 

\end{proof} 


 
 




\end{document}